\documentclass[11pt,reqno]{article}
\usepackage{amsmath,amssymb,amsthm,amsfonts,amstext,amsbsy,amscd}
\usepackage[utf8]{inputenc}
\usepackage{bbm,dsfont}
\usepackage{xcolor}
\usepackage{comment}
\usepackage[colorlinks,citecolor=blue,urlcolor=blue,linkcolor = blue,hypertexnames=false]{hyperref}

\theoremstyle{plain}
\newtheorem{thm}{Theorem}
\newtheorem{proposition}{Proposition}
\newtheorem{corollary}{Corollary}
\newtheorem{lemma}{Lemma}
\newtheorem{assumption}{Assumption}

\theoremstyle{definition}

\theoremstyle{remark}
\newtheorem{remark}{Remark}

\newtheorem{example}{Example}

\DeclareMathOperator{\R}{{\mathbb R}}

\DeclareMathOperator{\tr}{tr}

\DeclareMathOperator{\argmin}{arg\,min}

\newcommand{\SGN}{L}
\newcommand{\CEV}{C_{ev}}
\newcommand{\EstPCR}{\hat f_d^{PCR}}
\newcommand{\EstORA}{\hat f_d^{oracle}}

\renewcommand{\le}{\leqslant}
\renewcommand{\ge}{\geqslant}

\numberwithin{equation}{section}

\allowdisplaybreaks[4]

\begin{document}
\title{A note on the prediction error of principal component regression in high dimensions}
\author{Laura Hucker\thanks{Humboldt-Universit\"{a}t zu Berlin, Germany. E-mail: huckerla@math.hu-berlin.de} \qquad Martin Wahl\thanks{Universit\"{a}t Bielefeld, Germany. E-mail: martin.wahl@math.uni-bielefeld.de
\newline
}}
\date{}

\maketitle

\begin{abstract}
We analyze the prediction error of principal component regression (PCR) and prove high probability bounds for the corresponding squared risk conditional on the design. Our first main result shows that PCR performs comparably to the oracle method obtained by replacing empirical principal components by their population counterparts, provided that an effective rank condition holds. On the other hand, if the latter condition is violated, then empirical eigenvalues start to have a significant upward bias, resulting in a self-induced regularization of PCR. Our approach relies on the behavior of empirical eigenvalues, empirical eigenvectors and the excess risk of principal component analysis in high-dimensional regimes.
\end{abstract}


\section{Introduction}
Let $(\mathcal{H},\langle\cdot,\cdot\rangle)$ be a separable Hilbert space and let $(Y,X)$ be a pair of random variables satisfying the regression equation
\begin{equation}\label{EqModel}
Y=\langle f,X\rangle+\epsilon,
\end{equation}
where $f\in\mathcal{H}$, $X$ is a random variable with values in $\mathcal{H}$ such that $\mathbb{E}X=0$ and $\mathbb{E}\|X\|^2<\infty$, and $\epsilon$ is a real-valued random variable such that $\mathbb{E}(\epsilon\,|\,X)=0$ and $\mathbb{E}(\epsilon^2\,|\,X)=\sigma^2$. We suppose that we observe $n$ independent copies $(Y_1,X_1),\dots,(Y_n,X_n)$ of $(Y,X)$ and consider the problem of estimating $f$.

Allowing for general Hilbert spaces $\mathcal{H}$, the statistical model \eqref{EqModel} covers several regression problems. It includes a special case of the functional linear model, in which the responses are scalars and the covariates are curves. This model has been extensively studied in the literature (see, e.g., the monographs by Horv\'{a}th and Kokoszka \cite{HK12} and Hsing and Eubank \cite{HE15}). It also includes kernel-based learning problems. For instance, nonparametric regression with random design can be written in the form \eqref{EqModel}, provided that the regression function is contained in a reproducing kernel Hilbert space. This connection between statistical learning theory and the theory of linear inverse problems has been developed e.g.~in De Vito et al.~\cite{VRC05}. 

In this paper, we focus on estimating $f$ by principal component regression (PCR). PCR is a widely known estimation procedure, in which principal component analysis (PCA) is applied in a first step to reduce the high dimensionality of the data. Then, in a second step, an estimator of $f$ is obtained by regressing the responses on the leading empirical principal components. 

Our main goal is to show that the behavior of the prediction error of PCR crucially depends on an effective rank condition. On the one hand, if this condition is satisfied, then PCR performs comparably to the oracle estimator obtained by replacing empirical principal components by their population counterparts. On the other hand, if this condition is violated, then the eigenvalues of the empirical covariance operator start to have a significant upward bias, resulting in a regularization effect of the PCR estimator compared to the corresponding oracle estimator.

A general account on PCR is given in the monograph by Jolliffe \cite{J02}. For a study of the prediction error of PCR with focus on minimax optimal rates of convergence under standard model assumptions, see Hall and Horowitz \cite{HH07} and Brunel, Mas, and Roche \cite{BMR16} for the functional data context and Lu and Pereverzev \cite{MR3114700} and Blanchard and M\"{u}cke \cite{BM18} for the statistical learning context. Note that the latter study PCR (resp.~spectral cut-off) within a larger class of spectral regularization methods. In an overparametrized regime, the $n$th PCR estimator corresponds to the minimum-norm least squares estimator, which can predict well although it interpolates the data, see, e.g., Bartlett et al.~\cite{MR4263288}, Tsigler and Bartlett \cite{TB20} and Mei, Misiakiewicz, and Montanari \cite{MR4412180}. One explanation of this so-called Benign overfitting phenomenon is given by the fact that the overparametrized component results in a self-induced regularization (see, e.g., Bartlett, Montanari, and Rakhlin~\cite{MR4295218}). 

In settings where the number of observations is comparable to the dimension, it has been shown that high-dimensional phenomena occur in PCA, ranging from eigenvalue upward bias to eigenvector delocalization. While these phenomena are now well understood in the spiked covariance model (see, e.g., Nadler \cite{N08}, Paul \cite{P07},  Benaych-Georges and Nadakuditi \cite{MR2782201} and Bloemendal et al.~\cite{MR3449395}), the results obtained in this paper rely on nonasymptotic extensions to general eigenvalue settings (see, e.g.,~Rei\ss{} and Wahl \cite{RW17}, Jirak and Wahl \cite{MR4517351} and Bartlett et al.~\cite{MR4263288}).

\subsection{Basic notation}\label{BasicNotation} 

Let $\|\cdot\|$ denote the norm on $\mathcal{H}$.  Given a bounded (resp.~Hilbert-Schmidt) operator $A$ on $\mathcal{H}$, we denote the operator norm (resp.~the Hilbert-Schmidt norm) of $A$ by $\|A\|_\infty$ (resp.~$\|A\|_2$). Given a trace class operator $A$ on $\mathcal{H}$, we denote the trace of $A$ by $\tr(A)$. For $g\in\mathcal{H}$, the empirical and the $L^2(\mathbb{P}^X)$ norm are defined by $\|g\|_{n}^2=n^{-1}\sum_{i=1}^n\langle g,X_i\rangle^2$ and $\|g\|_{L^2(\mathbb{P}^X)}^2=\mathbb{E}\langle g,X\rangle^2$, respectively. For $g,h\in\mathcal{H}$, we denote by $g\otimes h$ the rank-one operator on $\mathcal{H}$ defined by $(g\otimes h)x=\langle h,x\rangle g$, $x\in \mathcal{H}$. We write $\mathbf{Y}=(Y_1,\dots,Y_n)^T$ and $\boldsymbol{\epsilon}=(\epsilon_1,\dots,\epsilon_n)^T$. Throughout the paper, $c\in (0,1)$ and $C>1$ denote constants that may change from line to line (by a numerical value). If no further dependencies are mentioned, then these constants are absolute. For a real-valued random variable $Z$ and $\alpha>0$, we define the norm $\|Z\|_{\psi_\alpha}=\sup_{p\ge 1}p^{-1/\alpha}\mathbb{E}^{1/p}|Z|^p$. For $\alpha\ge 1$, these norms are equivalent to the Orlicz norms.


\section{Main results}

In this section, we present our main bounds for the prediction error of principal component regression. All results are proved in Section \ref{SecProofs} below. 

\subsection{Principal component analysis}\label{SecPCA}

The covariance operator of $X$ is denoted by $\Sigma$. Since $X$ is assumed to be centered and strongly square-integrable (meaning that $\mathbb{E}X=0$ and $\mathbb{E}\|X\|^2<\infty$), the covariance operator $\Sigma$ can be defined by $\Sigma=\mathbb{E}(X\otimes X)$, where expectation is taken in the Hilbert space of all Hilbert-Schmidt operators on $\mathcal{H}$. It is well-known that under the above assumptions, the covariance operator $\Sigma$ is a positive, self-adjoint trace class operator (see, e.g., \cite[Theorem 7.2.5]{HE15}). By the spectral theorem, there exists a sequence $\lambda_1\ge \lambda_2\ge\dots>0$ of  positive eigenvalues (which is either finite or converges to zero) together with an orthonormal system of eigenvectors $u_1,u_2,\dots$ such that $\Sigma$ has the spectral decomposition
\begin{equation*}
\Sigma=\sum_{j\ge 1}\lambda_j P_j,\qquad P_j=u_j\otimes u_j.
\end{equation*}
Without loss of generality we shall assume that the eigenvectors $u_1,u_2,\dots$ form an orthonormal basis of $\mathcal{H}$ such that $\sum_{j\ge 1}P_j=I$. Moreover, let
\begin{equation*}
\hat{\Sigma}=\frac{1}{n}\sum_{i=1}^nX_i\otimes X_i
\end{equation*}
be the empirical covariance operator. Again, there exists a sequence $\hat{\lambda}_1\ge \hat{\lambda}_2\ge\dots\ge 0$ of eigenvalues together with an orthonormal basis of eigenvectors $\hat{u}_1,\hat{u}_2,\dots$ such that we can write
\begin{equation*}
\hat{\Sigma}=\sum_{j\ge 1}\hat{\lambda}_j \hat{P}_j,\qquad \hat{P}_j=\hat{u}_j\otimes\hat{u}_j.
\end{equation*}
For $d\ge 1$, we write
\[
P_{\le d}=\sum_{j\le d}P_j,\quad P_{>d}=\sum_{k>d}P_k, \qquad \hat{P}_{\le d}=\sum_{j\le d}\hat{P}_j,\quad \hat{P}_{> d}=\sum_{k> d}\hat{P}_k
\]
for the orthogonal projections onto the linear subspace spanned by the first $d$ eigenvectors of $\Sigma$ (resp.~$\hat \Sigma$) and onto its orthogonal complement. Similarly, we use the abbreviations $\Sigma_{>d}$, $\tr_{>d}(\Sigma)$ and $\tr_{>d}(\Sigma^2)$ for $\sum_{k>d} \lambda_k P_k$, $\sum_{k>d}\lambda_k$ and $\sum_{k>  d}\lambda_k^2$, respectively. Introducing 
\begin{equation*}
    {\mathcal P}_{d}=\{P:{\mathcal H}\to{\mathcal H}\,|\, P\text{ is orthogonal projection of rank }d\},
\end{equation*}
the reconstruction error of $P\in{\mathcal P}_{d}$ is defined by
\begin{equation*}
R(P)=\mathbb{E}\|X-PX\|^2.
\end{equation*}
The idea behind PCA is that $P_{\le d}$ satisfies
\begin{equation*}
P_{\le d}\in\argmin_{P\in {\mathcal P}_{d}}\limits R(P)\quad\text{and}\quad R(P_{\le d})=\tr_{>d}(\Sigma).
\end{equation*}
Similarly, the empirical reconstruction error of $P\in{\mathcal P}_{d}$ is defined by
\begin{equation*}
R_n(P)=\frac{1}{n}\sum_{i=1}^n \|X_i-PX_i\|^2,
\end{equation*}
and we have
\begin{equation*}
\hat{P}_{\le d}\in\argmin_{P\in {\mathcal P}_{d}}\limits R_n(P).
\end{equation*}
The excess risk of the PCA projector $\hat{P}_{\le d}$ is defined by
\begin{equation*}
{\mathcal E}^{PCA}_d=R(\hat{P}_{\le d})-\min_{P\in {\mathcal P}_{d}}R(P)=R(\hat{P}_{\le d})-R( P_{\le d}).
\end{equation*}
It has been shown in \cite[Lemma 2.6]{RW17} that, for every $\mu\in\R$, the excess risk of PCA can be represented as
\begin{align*}
{\mathcal E}^{PCA}_d={\mathcal E}^{PCA}_{\le d}(\mu)+{\mathcal E}^{PCA}_{> d}(\mu)
\end{align*}
with 
\begin{equation*}
{\mathcal E}^{PCA}_{\le d}(\mu)=\sum_{j\le d}(\lambda_j-\mu)\| P_j\hat{P}_{> d}\|_2^2
,\quad
{\mathcal E}^{PCA}_{> d}(\mu)=\sum_{k>d}(\mu-\lambda_{k})\| P_k\hat{P}_{\le d}\|_2^2.
\end{equation*} 
The excess risk is a natural measure of accuracy of the PCA projector $\hat P_{\le d}$ for many tasks like reconstruction and prediction (see, e.g., \cite{RW17}).

\subsection{Principal component regression}\label{SecPCR}

Let $\hat U_d=\operatorname{span}(\hat u_1,\dots,\hat u_d)$ be the linear subspace spanned by the first $d$ eigenvectors of $\hat \Sigma$. Then the PCR estimator in dimension $d$ is defined by
\begin{equation*}
\EstPCR\in\argmin_{g\in \hat U_d}\limits\|\mathbf{Y}-S_ng\|_n^2
\end{equation*}
with sampling operator $S_n:\mathcal{H}\rightarrow \mathbb{R}^n,h\mapsto (\langle h,X_i\rangle)_{i=1}^n$.
The PCR estimator can be defined explicitly by using the singular value decomposition of $S_n$. In fact, writing
\[
n^{-1/2}S_n=\sum_{j=1}^n\hat\lambda_j^{1/2}\hat v_j\otimes\hat u_j
\]
with orthonormal basis $\hat v_1,\dots, \hat v_n$ of $\mathbb{R}^n$ with respect to the Euclidean inner product on $\mathbb{R}^n$, also denoted by $\langle\cdot,\cdot\rangle$ in what follows (use that $n^{-1}S_n^*S_n=\hat\Sigma$ with adjoint operator $S_n^*:\mathbb{R}^n\rightarrow \mathcal{H}, v\mapsto \sum_{i=1}^nv_iX_i$), we have
\begin{equation}\label{EqRepPCREst}
\EstPCR=n^{-1/2}\sum_{j\le d}\hat\lambda_j^{-1/2}\langle \mathbf{Y},\hat{v}_j\rangle \hat u_j,
\end{equation}
provided that $\hat \lambda_d>0$. Our main goal is to analyze the squared prediction error of PCR given by
\begin{equation}\label{EqPredLoss}
\|\EstPCR-f\|_{L^2(\mathbb{P}^X)}^2=\langle \EstPCR-f,\Sigma(\EstPCR-f)\rangle.
\end{equation}
The following bias-variance decomposition of the mean squared prediction error of PCR conditional on the design is the starting point for our results.
\begin{lemma}\label{LemBVDecCond} If $\hat\lambda_ d> 0$, then we have
\begin{align*}
\mathbb{E}\big(\|\EstPCR-f\|_{L^2(\mathbb{P}^X)}^2\,\big|\,X_1,\dots,X_n\big)=B_d^2(f)+V_d
\end{align*}
with 
\begin{align*}
B_d^2(f)=\|\hat P_{>d}f\|_{L^2(\mathbb{P}^X)}^2\quad\text{and}\quad V_d=\frac{\sigma^2}{n}\sum_{j\le d}\hat\lambda_j^{-1}\|\hat u_j\|_{L^2(\mathbb{P}^X)}^2.
\end{align*}
\end{lemma}

\subsection{The oracle estimator}

Letting $U_d=\operatorname{span}(u_1,\dots,u_d)$, we define the oracle PCR estimator in dimension $d$ by
\[
\EstORA\in\argmin_{g\in  U_d}\limits\|\mathbf{Y}-S_ng\|_n^2.
\] 
The oracle PCR estimator has an analogous decomposition as in \eqref{EqRepPCREst}, with empirical eigenvalues and eigenvectors replaced by those corresponding to the projected observations $P_{\le d}X_i$, cf.~\eqref{EqDecOracleEst} below. 
 
\begin{assumption}\label{SubGauss}
Suppose that $X$ is sub-Gaussian, meaning that there is a constant $\SGN$ such that
\[
\forall g\in\mathcal{H},\quad \|\langle g,X\rangle\|_{\psi_2} \le \SGN\mathbb{E}^{1/2}\langle g,X\rangle^2.
\]
\end{assumption}
The oracle PCR estimator satisfies the following standard risk bound, see, e.g., \cite[Chapter 11]{GKKW02} for similar results for the linear least squares estimator in random design or~\cite{WahlCelisse}.

\begin{proposition}\label{PropOracleRisk} Grant Assumption \ref{SubGauss}. Then there are constants $c,C>0$ depending only on $\SGN$ such that
\[
\forall d\le cn,\quad\mathbb{E}(\|\EstORA-f\|_{L^2(\mathbb{P}^X)}^2\,|\,X_1,\dots,X_n)\le C\Big(\sum_{k>d}\lambda_k\langle f,u_k\rangle^2+\frac{\sigma^2d}{n}\Big)
\]
with probability at least $1-Ce^{-cn}$.  
\end{proposition}
\begin{remark}
Proposition \ref{PropOracleRisk} only states an upper bound, but the proof also provides a corresponding lower bound.
\end{remark}

The oracle bias term $\sum_{k>d}\lambda_k\langle f,u_k\rangle^2$ is affected by the fact that $U_d$ does not necessarily have good approximation properties for $f$. In the extreme case, it is equal to $\lambda_{d+1}\|f\|^2$. 
 
\begin{corollary}\label{CorOracleRisk} Grant the assumptions of Proposition \ref{PropOracleRisk}. 
Then we have 
\[
\forall d\le cn,\quad\mathbb{E}(\|\EstORA-f\|_{L^2(\mathbb{P}^X)}^2\,|\,X_1,\dots,X_n)\le C\Big(\lambda_{d+1}\|f\|^2+\frac{\sigma^2d}{n}\Big)
\]
with probability at least $1-Ce^{-cn}$, where $c,C$ are the constants in Proposition~\ref{PropOracleRisk}. 
\end{corollary}

Corollary \ref{CorOracleRisk} can be improved under additional source conditions (see, e.g., \cite{MR3114700,BM18}). Yet, in order to keep this note short, the worst-case scenario in Corollary \ref{CorOracleRisk} will serve as a benchmark later on.

\subsection{Main result in the classical regime}

Our first main result shows that if $d\le cn$ and $\tr_{>d}(\Sigma)/\lambda_{d+1}\le cn$, then PCR performs comparably to the oracle estimator from the previous section. 

\begin{thm}\label{MainResultClassical}
Grant Assumption \ref{SubGauss}. Then there are constants $c,C>0$ depending only on $\SGN$ such that the following holds. If $d\ge  1$ satisfies
\begin{align}\label{MainResultClassicalCond}
\max\Big(d,\frac{\tr_{>d}(\Sigma)}{\lambda_{d+1}}\Big)\le cn,
\end{align}
then we have
\begin{align*}
\mathbb{E}\big(\|\EstPCR-f\|_{L^2(\mathbb{P}^X)}^2\,\big|\,X_1,\dots,X_n\big)&\le C\Big(\lambda_{d+1}\|f\|^2+\frac{\sigma^2d}{n}\Big)
\end{align*}
with probability at least $1-Ce^{-cn}$.
\end{thm}

Let us briefly illustrate Theorem \ref{MainResultClassical} for two standard eigenvalue decays from kernel-based learning and functional data analysis. 

\begin{example}[Exponential decay]
Suppose that there are constants $\alpha>0$ and $\CEV\ge 1$ such that
\begin{equation*}
\forall j\ge 1,\qquad \CEV^{-1}e^{-\alpha j}\le \lambda_j\le \CEV e^{-\alpha j}.
\end{equation*} 
Moreover, suppose that Assumption \ref{SubGauss} holds. Then, by Theorem \ref{MainResultClassical}, there exist constants $c,C>0$ depending only on $\SGN$, $\alpha$, $\CEV$, $\|f\|^2$ and $\sigma^2$ such that
\[
\forall d\le cn,\quad\mathbb{E}\big(\|\EstPCR-f\|_{L^2(\mathbb{P}^X)}^2\,\big|\,X_1,\dots,X_n\big)\le C\Big(e^{-\alpha d}+\frac{d}{n}\Big)
\]
with probability at least $1-Ce^{-cn}$. Choosing $d$ of size $\log (n)/\alpha$, we get the minimax optimal rate $\log (n)/n$ (see, e.g., \cite{CJ10}).
\end{example}

\begin{example}[Polynomial decay]
Suppose that there are constants $\alpha>1$ and $\CEV\ge 1$ such that
\begin{equation*}
\forall j\ge 1,\qquad \CEV^{-1}j^{-\alpha}\le \lambda_j\le \CEV j^{-\alpha}.
\end{equation*} 
Moreover, suppose that Assumption \ref{SubGauss} holds. Then, by Theorem \ref{MainResultClassical}, there exist constants $c,C>0$ depending only on $\SGN$, $\alpha$, $\CEV$, $\|f\|^2$ and $\sigma^2$ such that
\begin{align*}
\forall d\le cn,\quad\mathbb{E}\big(\|\EstPCR-f\|_{L^2(\mathbb{P}^X)}^2\,\big|\,X_1,\dots,X_n\big)\le C\Big(d^{-\alpha}+\frac{d}{n}\Big)
\end{align*}
with probability at least $1-Ce^{-cn}$. Choosing $d$ of size $n^{1/(\alpha+1)}$, we get the minimax optimal rate $n^{-\alpha/(\alpha+1)}$ (see, e.g., \cite{HH07,CJ10} or \cite{BM18}). 
\end{example}

\subsection{Main result in the high-dimensional regime}

In this section, we show that if $d\le n$ satisfies $\tr_{>d}(\Sigma)/\lambda_{d+1}\ge  Bn$ for an appropriate constant $B>1$ depending only on $L$, then high-dimensional phenomena of PCA result in a regularization effect of the PCR estimator. For our second main result, we introduce the Karhunen-Loève coefficients given by $\eta_j=\lambda_j^{-1/2}\langle X,u_j\rangle$, $j\ge 1$.

\begin{thm} \label{BenignOverfittingResult}
Grant Assumption \ref{SubGauss} and suppose that the Karhunen-Loève coefficients $\eta_1,\eta_2,\dots$ are independent. Then there are constants $B>1$ and $c,C>0$  depending only on $\SGN$ such that the following holds. Let 
\begin{align*}
    j^* = \min \Big\{ j \ge 0 \ \Big|\ \frac{\tr_{>j}(\Sigma)}{\lambda_{j+1}} \ge B n \Big\}.
\end{align*}
If $j^* \le cn$, then, for all $j^* \le d \le n$ and $1 \le t \le n$, we have
\begin{align*}
    &\mathbb{E}\big(\|\hat{f}_d^{PCR}-f\|_{L^2(\mathbb{P}^X)}^2\,\big|\,X_1,\dots,X_n\big) \\
    &\le C \bigg(\frac{\tr_{>j^*}(\Sigma)}{n}\|f\|^2 + t^2\frac{\sigma^2 j^*}{n} + \frac{\sigma^2}{\tr_{>j^*}(\Sigma)} \sum_{j=j^*+1}^d \tr(\hat{P}_j \Sigma_{>j^*})\bigg)
\end{align*}
with probability at least $1-Ce^{-ct}$. Moreover, for $d=n$, we have
\begin{align*}
    \sum_{j=j^*+1}^n \tr(\hat{P}_j \Sigma_{>j^*}) \le Cn \frac{\tr_{>j^*}(\Sigma^2)}{\tr_{>j^*}(\Sigma)} 
\end{align*}
with probability at least $1-Ce^{-cn}$.
\end{thm}

Let us illustrate Theorem \ref{BenignOverfittingResult} in two examples. Note that polynomial or exponential decay are both covered by the classical regime in Theorem \ref{MainResultClassical}.

\begin{example}[The spiked covariance model]
Grant Assumption \ref{SubGauss}. Suppose that $\mathcal{H}=\mathbb{R}^p$, that the Karhunen-Lo\`eve coefficients $\eta_1,\eta_2,\dots, \eta_p$ are independent and that
\begin{align*}
\lambda_1\ge  \dots\ge  \lambda_r>\lambda_{r+1}=\dots=\lambda_p.
\end{align*}
We assume that $r\le cn$ and $p\ge  Cn$. Applying the first part of Theorem~\ref{BenignOverfittingResult}, there are constants $c,C>0$ depending only on $\SGN$ and $B$ such that the following holds. If $j^*= r$,
then, for all $r\le d\le n$ and all $1 \le t \le n$, we have 
\begin{align*}
\mathbb{E}\big(\|\EstPCR-f\|_{L^2(\mathbb{P}^X)}^2\,\big|\,X_1,\dots,X_n\big) \le C\Big(\frac{\tr_{>r}(\Sigma)}{n}\|f\|^2+t^2\frac{\sigma^2r}{n}+\frac{\sigma^2(d-r)}{p}\Big)
\end{align*}
with probability at least $1-Ce^{-ct}$. To see this, note that
\begin{align*}
\sum_{j=r+1}^d \tr(\hat{P}_j \Sigma_{>r})=\lambda_{r+1}\sum_{j=r+1}^d \tr(\hat{P}_j P_{>r}) \le \lambda_{r+1}(d-r).
\end{align*}
Hence, in this special case, the variance of the PCR estimator is piecewise linear in $d$ and has a change point at $d=r+1$, where the slope decreases from $\sigma^2/n$ to $\sigma^2/p$. This phenomenon is explained by the eigenvalue upward bias presented in Corollary \ref{CorBartlett}.
\end{example}

\begin{example}[Benign overfitting phenomenon]
\label{ExpBartlett}
Grant the assumptions of Theorem \ref{BenignOverfittingResult}. Suppose that the Karhunen-Lo\`eve coefficients $\eta_1,\eta_2,\dots$ are independent and that $d=n$. Then Corollary~\ref{CorBartlett} implies that $\hat\lambda_n>0$ with high probability, meaning that $\hat{f}_n^{PCR}$ coincides with the minimum-norm least squares estimator 
\begin{align*}
\hat{f}_n^{PCR}=\argmin_{f\in\mathcal{H}:S_nf=\mathbf{Y}}\limits\|f\|,
\end{align*}
which interpolates the data:
\begin{align*}
    \hat{f}_n^{PCR}=S_n^*(S_nS_n^*)^{-1}\mathbf{Y}\quad\text{such that}\quad S_n\hat{f}_n^{PCR}=\mathbf{Y}.
\end{align*}
By Theorem \ref{BenignOverfittingResult}, we get, for all $1 \le t \le n$, 
\begin{align*}
    &\mathbb{E}\big(\|\hat{f}_n^{PCR}-f\|_{L^2(\mathbb{P}^X)}^2\,\big|\,X_1,\dots,X_n\big) \\ &\le C \bigg(\frac{\tr_{>j^*}(\Sigma)}{n}\|f\|^2 + t^2\frac{\sigma^2 j^*}{n} + \sigma^2 n \frac{\tr_{>j^*}(\Sigma^2)}{(\tr_{>j^*}(\Sigma))^2} \bigg)
\end{align*}
with probability at least $1-Ce^{-ct}$. Note that we get the same variance term as in Theorem 1 in \cite{MR4263288} and a similar bias term as in \cite{TB20}, provided that the worst-case scenario from Corollary \ref{CorOracleRisk} is considered. Note that we prove these results by different arguments based on perturbation bounds for the empirical covariance operator. This provides an alternative explanation of the self-induced regularization property of $\hat{f}_n^{PCR}$ from \cite{MR4295218}, namely through high-dimensional phenomena such as the eigenvalue upward bias and the eigenvector delocalization.
\end{example}


\section{Main tools} \label{SecTools}

In this section, we state perturbation bounds for empirical eigenvalues, empirical eigenvectors and the excess risk of PCA, developed in \cite{RW17,WahlCelisse,MR4263288}. These results form the basis of our analysis of the bias and the variance term in Lemma \ref{LemBVDecCond}. 

\subsection{Bounds for empirical eigenvalues}

\begin{lemma}\label{ThmEvDev} Grant Assumption \ref{SubGauss}. Then there are constants $c_1,c_2\in(0,1)$ depending only on $\SGN$ such that, for all $y>0$ and all $j\ge  1$ satisfying
\[
\frac{1}{y\wedge 1}\sum_{k> j}\frac{\lambda_k}{\lambda_j-\lambda_k+y\lambda_j}\le c_1n,
\]
we have
\begin{equation*}
\mathbb{P}\big(  \hat{\lambda}_{j}-\lambda_j> y\lambda_j\big)\le e^{ 1-c_2n (y\wedge y^2) }.
\end{equation*}
Moreover,  for all $y>0$ and $j\ge  1$ satisfying
\begin{equation*}
\frac{1}{y\wedge 1}\sum_{l< j}\frac{\lambda_l}{\lambda_l-\lambda_j+y\lambda_j}\le c_1n,
\end{equation*}
we have
\begin{equation*}
\mathbb{P}\big(  \hat{\lambda}_{j}-\lambda_j< - y\lambda_j\big)\le e^{ 1-c_2n (y\wedge y^2) }.
\end{equation*}
\end{lemma}

\begin{proof}
See Theorem  2.15 in \cite{RW17}.
\end{proof}

\begin{corollary}\label{CorEVConcLB}
If  Assumption \ref{SubGauss} holds, then we have
\begin{equation*}
\forall j\le c_1n/4, \quad\mathbb{P}(\hat\lambda_j\ge   \lambda_j/2)\ge  1- e^{1-c_2n/4},
\end{equation*}
where $c_1$ and $c_2$ are the constants in Lemma \ref{ThmEvDev}.
\end{corollary}

\begin{proof}
Corollary \ref{CorEVConcLB} follows from the second claim in Lemma \ref{ThmEvDev} applied with $y=1/2$.
\end{proof}

\begin{corollary}\label{CorEVConcUB}
Grant Assumption \ref{SubGauss}. Let $j\ge  1$ be such that  
\begin{align*}
\frac{\operatorname{tr}_{>  j-1}(\Sigma)}{\lambda_j}\ge  c_1n.
\end{align*}
Then, with probability at least $1-e^{1-c_2n}$, we have
\begin{align*}
\hat\lambda_{j}\le \frac{2}{c_1}\frac{\operatorname{tr}_{>  j-1}(\Sigma)}{n},
\end{align*}
where $c_1$ and $c_2$ are the constants in Lemma \ref{ThmEvDev}.
\end{corollary}

\begin{proof}
For $y\ge  1$, we have 
\begin{align*}
\frac{1}{y\wedge 1}\sum_{k> j}\frac{\lambda_k}{\lambda_j-\lambda_k+y\lambda_j}\le \frac{1}{y}\frac{\operatorname{tr}_{>  j-1}(\Sigma)}{\lambda_j}.
\end{align*}
Hence, the first assumption in Lemma \ref{ThmEvDev} is satisfied for
\begin{align*}
y=\frac{1}{c_1n}\frac{\operatorname{tr}_{>  j-1}(\Sigma)}{\lambda_j}\ge  1.
\end{align*}
We conclude that 
\begin{align*}
\hat\lambda_{j}\le \lambda_j+\frac{1}{c_1}\frac{\operatorname{tr}_{>  j-1}(\Sigma)}{n}\le \frac{2}{c_1}\frac{\operatorname{tr}_{>  j-1}(\Sigma)}{n}
\end{align*}
with probability at least $1-e^{1-c_2n}$.
\end{proof}

The following lemma is a version of Lemma 4 in \cite{MR4263288}. It is the only point that requires the independence of the Karhunen-Loève coefficients $\eta_j=\lambda_j^{-1/2}\langle X,u_j\rangle$, $j\ge 1$.

\begin{lemma}\label{LemBartlett}
Grant Assumption \ref{SubGauss} and suppose that the Karhunen-Loève coefficients $\eta_1$, $\eta_2$, $\dots$ are independent. Then there exist constants $c_2\in(0,1)$ and $C_1>1$ depending only on $L$ such that the following holds. If
\begin{align*}
\frac{\operatorname{tr}_{> n}(\Sigma)}{\lambda_{n+1}}\ge  C_1n,
\end{align*}
then
\begin{align*}
\hat\lambda_{n}\ge  \frac{1}{2}\frac{\operatorname{tr}_{> n}(\Sigma)}{n}
\end{align*}
with probability at least $1-e^{1-c_2n}$.
\end{lemma}

\begin{proof}
Set $X_{i}'=P_{> n}X_i$, $1\le i\le n$, and let $K_n'=n^{-1}(\langle X_i',X_{i'}'\rangle)_{i,i'=1}^n$ be the corresponding normalized Gram matrix. Using that the $n$ largest eigenvalues of the empirical covariance operator $\hat\Sigma=n^{-1}S_n^*S_n$ coincide with the $n$ eigenvalues of the normalized Gram matrix $K_n=n^{-1}S_nS_n^*$ and that $K_n- K_n'$ is positive semi-definite, we have 
\begin{align*}
\hat\lambda_n=\lambda_n(\hat\Sigma)=\lambda_n(K_n)\ge  \lambda_n(K_n'),
\end{align*}
where $\lambda_n(\cdot)$ denotes the $n$th largest eigenvalue. Since the expectation of $K_n'$ is equal to $(\operatorname{tr}_{>n}(\Sigma)/n)\cdot I_n$, we get 
\begin{align*}
\hat\lambda_n\ge  \frac{\operatorname{tr}_{>n}(\Sigma)}{n}-\|K_n'-\mathbb{E}K_n'\|_{\infty}.
\end{align*}
Note that $(K_n')_{ii'} = n^{-1}\sum_{j>n} \lambda_j \eta_{ji} \eta_{ji'}$, where the $\eta_{ji}=\lambda_j^{-1/2} \langle X_i,u_j \rangle$, $j \ge 1$, $1\le i\le n$, are independent with $\sup_{j\ge 1} \| \eta_{ji} \|_{\psi_2} \le L$ for all $i\ge 1$. Hence, combining Bernstein's inequality with an $\varepsilon$-net argument to control the operator norm, we get
\begin{align*}
\|K_n'-\mathbb{E}K_n'\|_{\infty}&\le C\frac{\sqrt{\sum_{k>n}\lambda_k^2}}{\sqrt{n}}+C\lambda_{n+1}\\
&\le C\sqrt{\lambda_{n+1}}\frac{\sqrt{\operatorname{tr}_{>n}(\Sigma)}}{\sqrt{n}}+C\lambda_{n+1}\le \frac{\operatorname{tr}_{>n}(\Sigma)}{4n}+(C^2+C)\lambda_{n+1}
\end{align*}
with probability at least $1-e^{1-c_2n}$. For more details, see the proof of Lemma 4 in \cite{MR4263288}.
The claim now follows, provided that $C^2+C\le C_1/4$.
\end{proof}

\begin{corollary} \label{CorBartlett}
Grant Assumption \ref{SubGauss} and suppose that the Karhunen-Loève coefficients $\eta_1,\eta_2,\dots$ are independent. Let $j^*$ be defined as in Theorem \ref{BenignOverfittingResult} with $B\ge C_1+1$. If $j^* < n$, then
\begin{align*}
    \frac{1}{2}\frac{\operatorname{tr}_{> n}(\Sigma)}{n}\le \hat\lambda_{n}\le \cdots \le \hat\lambda_{j^*+1}\le \frac{4}{c_1}\frac{\operatorname{tr}_{> n}(\Sigma)}{n}
\end{align*}
with probability at least $1-2e^{1-c_2n}$.  Here, $c_1$ is the constant from Corollary \ref{CorEVConcUB} and $C_1$ is the constant from Lemma~\ref{LemBartlett}, while $c_2$ is the minimum of the constants from Corollary \ref{CorEVConcUB} and Lemma~\ref{LemBartlett}.
\end{corollary}

\begin{proof}
Since $B\ge C_1+1$, we have 
\begin{align}
    \frac{\operatorname{tr}_{> n}(\Sigma)}{\lambda_{j^*+1}}&\ge \frac{\operatorname{tr}_{> j^*}(\Sigma)}{\lambda_{j^*+1}}-n\ge \Big(1-\frac{1}{C_1+1}\Big)\frac{\operatorname{tr}_{> j^*}(\Sigma)}{\lambda_{j^*+1}}\ge \frac{1}{2}\frac{\operatorname{tr}_{> j^*}(\Sigma)}{\lambda_{j^*+1}},\label{eq:change:remainder:trace}\\
    \frac{\operatorname{tr}_{> n}(\Sigma)}{\lambda_{n+1}}&\ge \frac{\operatorname{tr}_{> j^*}(\Sigma)}{\lambda_{j^*+1}}-n\ge (C_1+1)n-n= C_1n,\nonumber
\end{align}
as can be seen from the definition of $j^*$. Hence, Lemma \ref{LemBartlett} gives the lower bound, while the upper bound follows from Corollary \ref{CorEVConcUB} applied with $j=j^*+1$ and the bound $\operatorname{tr}_{> j^*}(\Sigma)\le 2\operatorname{tr}_{> n}(\Sigma)$. 
\end{proof}

\subsection{Change of norm argument} 

In this section, we present a change of norm formula.

\begin{lemma}\label{LemChangeNorm1.main} 
Let $\mu\ge  0$. Then we have
\begin{align*}
\forall h\in\mathcal{H},\quad |\|h\|_n^2-\|h\|_{L^2(\mathbb{P}^X)}^2|\le (1/2) (\|h\|_{L^2(\mathbb{P}^X)}^2+\mu\|h\|^2)
\end{align*}
if and only if 
\begin{align}\label{eq_changeofnorm_event}
\|(\Sigma+\mu)^{-1/2}(\hat\Sigma-\Sigma)(\Sigma+\mu)^{-1/2}\|_{\operatorname{\infty}}\le 1/2.
\end{align}
Moreover, if \eqref{eq_changeofnorm_event} holds, then $\hat\lambda_j\le 3\lambda_j/2+\mu/2$ for all $j\ge  1$.
\end{lemma}

\begin{proof}
See Lemmas 16 and 44 in \cite{WahlCelisse}.
\end{proof}

\begin{lemma}\label{LemConcIneq}
Grant Assumption \ref{SubGauss} and let $\mu\ge 0$ (with $\mu>0$ if $\mathcal{H}$ is infinite-dimensional). Then there are constants $c_1,c_2\in (0,1)$ depending only on $\SGN$ such that 
\[
\mathbb{P}(\|(\Sigma+\mu)^{-1/2}(\hat\Sigma-\Sigma)(\Sigma+\mu)^{-1/2}\|_\infty>1/2)\le e^{-c_2n},
\]
whenever
\begin{align}\label{changenormcondition}
\mathcal{N}(\mu):=\sum_{j\ge  1}\frac{\lambda_j}{\lambda_j+\mu} \le c_1n.
\end{align}
\end{lemma}

\begin{proof}
The claim follows from Theorem 9 in \cite{KL14} (see, e.g., Lemma 3.9 in \cite{RW17} for a similar statement). The main point is that the transformed observations $(\Sigma+\mu)^{-1/2}X_i$, $1\le i\le n$, also satisfy Assumption \ref{SubGauss} (with the same constant $L$). Moreover, the largest eigenvalue and the trace of the transformed covariance operator are given by $\lambda_1/(\lambda_1+\mu)\in (0,1]$ and $\mathcal{N}(\mu)$, respectively.
\end{proof}

\begin{remark}\label{RemBartlettResult}
The quantity $\mathcal{N}(\mu)$ is also called effective dimension (see, e.g.,~\cite{BM18}). 
\end{remark}

\begin{corollary}\label{CorChangeNormFunction}
Grant Assumption \ref{SubGauss}. Let $d\ge 1$. If $d \le c_1n/2$,
then we have
\begin{align*}
\forall h\in\mathcal{H},\quad \|h\|_{L^2(\mathbb{P}^X)}^2\le 2\|h\|_n^2+ 2 \frac{\tr_{>d}(\Sigma)}{c_1n}\|h\|^2
\end{align*}
with probability at least $1-e^{-c_2n}$, where $c_1$ and $c_2$ are the constants in Lemma \ref{LemConcIneq}.
\end{corollary}

\begin{proof}
Observe that
\[
 \mathcal{N}(\mu)\le d+\frac{\operatorname{tr}_{>d}(\Sigma)}{\mu} \le \frac{c_1n}{2}+\frac{\operatorname{tr}_{>d}(\Sigma)}{\mu}.
\] 
Hence, \eqref{changenormcondition} is satisfied for the choice 
\begin{align}\label{CorChangeNormFunctionChoiceMu}
    \mu = 2\frac{\tr_{>d}(\Sigma)}{c_1n},
\end{align}
and the claim follows from combining Lemmas \ref{LemChangeNorm1.main}  and \ref{LemConcIneq}.
\end{proof}

\begin{corollary}\label{CorChangeNormEV}
Grant Assumption \ref{SubGauss}. Let $d\ge 1$.  If $d \le c_1n/2$, then we have
\begin{align*}
\forall j\ge  1,\quad \|\hat u_j\|_{L^2(\mathbb{P}^X)}^2& \le 3\lambda_j+ 4 \frac{\tr_{>d}(\Sigma)}{c_1n}\text{ and }\hat\lambda_j \le 3\lambda_j/2+\frac{\tr_{>d}(\Sigma)}{c_1n}
\end{align*}
with probability at least $1-e^{-c_2n}$, where $c_1$ and $c_2$ are the constants in Lemma \ref{LemConcIneq}.
\end{corollary}

\begin{proof}
With the choice in \eqref{CorChangeNormFunctionChoiceMu}, the claim for the eigenvalues follows from Lemma \ref{LemConcIneq} and the second claim in Lemma \ref{LemChangeNorm1.main}. Combining this with Corollary~\ref{CorChangeNormFunction}, we get
\begin{align*}
\forall j\ge  1,\quad \|\hat u_j\|_{L^2(\mathbb{P}^X)}^2&\le 2\hat\lambda_j+ 2 \frac{\tr_{>d}(\Sigma)}{c_1n} \le 3\lambda_j+ 4 \frac{\tr_{>d}(\Sigma)}{c_1n}
\end{align*}
with probability at least $1-e^{-c_2n}$.
\end{proof}

\subsection{Bounds for the excess risk}\label{SecRW17}

In this section, we state a perturbation bound for the excess risk obtained in \cite{RW17}. 

\begin{lemma}\label{LemExcessRiskHighProbability}
Grant Assumption \ref{SubGauss}. Then there are constants $c_1,c_2,C_1>0$ depending only on $\SGN$ such that the following holds. Let $d\ge  1$, $\mu \ge \lambda_{d+1}$, and let $r\le d$ be the largest integer such that $\lambda_r>\mu$ (with $r=0$ if $\mu\ge \lambda_1$). If
\begin{equation}\label{ThmERBoundCond}
\frac{\lambda_r}{\lambda_r-\lambda_{d+1}}\max\Big(\sum_{j\le r}\frac{\lambda_j}{\lambda_j-\lambda_{d+1}},\sum_{k>d}\frac{\lambda_k}{\lambda_r-\lambda_k}\Big) \le c_1n,
\end{equation}
then, for any $t\ge  1$, we have
\begin{align*}
\mathcal E^{PCA}_{\le d}(\mu)\le C_1\frac{t^2}{n}\sum_{j\le r}\frac{\lambda_j\operatorname{tr}_{>r}(\Sigma)}{\lambda_j-\lambda_{d+1}}
\end{align*}
with probability at least $1-2e^{1-c_2n(\lambda_r-\lambda_{d+1})^2/\lambda_r^2}-e^{-t}$.
\end{lemma}

Lemma \ref{LemExcessRiskHighProbability} is a high probability version of Proposition 2.10 in \cite{RW17}. Its proof is given in Section \ref{ProofLemExcessRiskHighProbability}.

\begin{corollary} \label{CorExcessRisk0}
Grant Assumption \ref{SubGauss}. Let $j^*$ be defined as in Theorem \ref{BenignOverfittingResult}. Suppose that $j^* \le c_1n/4$. Then, for any $t \ge 1$, we have
\begin{align*}
    {\mathcal E}^{PCA}_{\le j^*}(0) \le C t^2 \frac{j^* \tr_{>j^*}(\Sigma)}{n}
\end{align*}
with probability at least $1-2e^{1-c_2n/4}-e^{-t}$, where $C=4C_1+4/c_1$ and $c_1,c_2$ and $C_1$ are the constants in Lemma \ref{LemExcessRiskHighProbability}.
\end{corollary}

\begin{proof}
Observe that
\begin{align}
    \mathcal E^{PCA}_{\le j^*}(0) \le \mathcal E^{PCA}_{\le j^*}(\mu) + \mu \tr(P_{\le j^*} \hat P_{>j^*}) \le \mathcal E^{PCA}_{\le j^*}(\mu) + \mu j^*. \label{EqDecompExcessRisk0}
\end{align}
In order to apply Lemma \ref{LemExcessRiskHighProbability}, we need to choose $\mu$ such that \eqref{ThmERBoundCond} is satisfied. First, note that if
\begin{align}\label{eq:lower:bound:mu}
    \mu\ge 2\lambda_{j^*+1},
\end{align}
then
\begin{align*}
    \frac{\lambda_r}{\lambda_r-\lambda_{j^*+1}}&\le \frac{2\lambda_{j^*+1}}{2\lambda_{j^*+1}-\lambda_{j^*+1}}=2,\\
    \sum_{j \le r} \frac{\lambda_j}{\lambda_j-\lambda_{j^*+1}}&\le j^*\frac{2\lambda_{j^*+1}}{2\lambda_{j^*+1}-\lambda_{j^*+1}} = 2j^*\le \frac{c_1n}{2},\\
    \sum_{k>j^*} \frac{\lambda_k}{\lambda_r-\lambda_k} &\le \frac{\tr_{>j^*}(\Sigma)}{\lambda_r-\lambda_{j^*+1}}\le 2\frac{\tr_{>j^*}(\Sigma)}{\mu}.
\end{align*}
Choosing 
\begin{align*}
\mu = 4\frac{\tr_{>j^*}(\Sigma)}{c_1n},
\end{align*}
we get that \eqref{ThmERBoundCond} and \eqref{eq:lower:bound:mu} are satisfied for $d=j^*$, where the latter can be seen from $\tr_{>j^*}(\Sigma)\ge Bn \lambda_{j^*+1}$ and the fact that $B>1$ and $c_1\in(0,1)$. Hence, combining \eqref{EqDecompExcessRisk0} with Lemma \ref{LemExcessRiskHighProbability}, we conclude that, for all $t\ge 1$,
\begin{align*}
    \mathcal E^{PCA}_{\le j^*}(0)\le 2C_1t^2\frac{j^*\operatorname{tr}_{>r}(\Sigma)}{n}+4\frac{j^*\tr_{>j^*}(\Sigma)}{c_1n}
\end{align*}
with probability at least $1-2e^{1-c_2n/4}-e^{-t}$.
Inserting 
\begin{align*}
    \tr_{>r}(\Sigma) &\le (j^*-r)\mu + \tr_{>j^*}(\Sigma) = \Big(\frac{4(j^*-r)}{c_1n}+1\Big) \tr_{>j^*}(\Sigma) \le 2 \tr_{>j^*}(\Sigma),
\end{align*}
the claim follows.
\end{proof}


\section{Proofs}\label{SecProofs}

\subsection{Proof of Lemma \ref{LemBVDecCond}}\label{LemBVDecCondProof}

Inserting $\mathbf{Y}=S_nf+\boldsymbol{\epsilon}$ into \eqref{EqRepPCREst}, we have
\begin{equation}\label{EqDecEstPCR}
\EstPCR=\hat P_{\le d}f+n^{-1/2}\sum_{j\le d}\hat\lambda_j^{-1/2}\langle \boldsymbol{\epsilon},\hat{v}_j\rangle \hat u_j,
\end{equation}
provided that $\hat \lambda_d>0$. Inserting \eqref{EqDecEstPCR} and the identity $f=\hat P_{\le d}f+\hat P_{> d}f$ into \eqref{EqPredLoss}, we get
\begin{align*}
\|\EstPCR-f\|_{L^2(\mathbb{P}^X)}^2&=\langle \hat P_{>d}f,\Sigma \hat P_{>d}f\rangle
-2n^{-1/2}\sum_{j\le d}\hat\lambda_j^{-1/2}\langle \hat P_{> d}f, \Sigma \hat u_j\rangle\langle \boldsymbol{\epsilon},\hat{v}_j\rangle\\
&+n^{-1}\sum_{j\le d}\sum_{k\le d}\hat\lambda_j^{-1/2}\hat\lambda_k^{-1/2}\langle\hat u_j,\Sigma \hat u_k\rangle\langle \boldsymbol{\epsilon},\hat{v}_j\rangle\langle \boldsymbol{\epsilon},\hat{v}_k\rangle.
\end{align*}
The result now follows from the fact that, conditional on the design, the $\langle \boldsymbol{\epsilon}, \hat v_j\rangle$ are uncorrelated, each with expectation zero and variance $\sigma^2$.\qed

\subsection{Proof of Proposition \ref{PropOracleRisk}}\label{PropOracleRiskProof}

We abbreviate $\EstORA$ by $\hat f_d$. Consider $X'=P_{\le d}X$, $X_i'=P_{\le d}X_i$ (defined on $U_d$), which lead to covariance and sample covariance $\Sigma'=P_{\le d}\Sigma P_{\le d}$, $\hat\Sigma'=P_{\le d}\hat\Sigma P_{\le d}$. In the following, we use the notation introduced in Sections~\ref{SecPCA} and \ref{SecPCR} with an additional superscript $'$. With this notation, we have 
\begin{equation}\label{EqDecOracleEst}
\hat f_d=n^{-1/2}\sum_{j\le d}\hat\lambda_j'^{-1/2}\langle \mathbf{Y},\hat{v}_j'\rangle \hat u_j',
\end{equation}
provided that $\hat\lambda_d'>0$. We define the events $\mathcal{A}_d=\{\hat\lambda_d'\ge \lambda_d/2\}$ and
\[
\mathcal{E}_d=\{(1/2)\|g\|_{L^2(\mathbb{P}^X)}^2\le \|g\|_n^2\le (3/2)\|g\|_{L^2(\mathbb{P}^X)}^2\,\forall g\in U_d\}.
\] 
Let $c_1$ and $c_2$ be the minima of the respective constants in Lemmas \ref{ThmEvDev} and \ref{LemConcIneq}. By Lemmas \ref{LemChangeNorm1.main} and \ref{LemConcIneq} (applied to the observations $X_i'$ taking values in the $d$-dimensional Hilbert space $U_d$ and $\mu=0$), we have $\mathbb{P}(\mathcal{E}_d^c)\le e^{-c_2n}$, provided that $d\le c_1n$. Moreover, by Corollary \ref{CorEVConcLB}, we have $\mathbb{P}(\mathcal{A}_d^c)\le e^{1-c_2n/4}$, provided that $d\le c_1n/4$. Hence, it remains to analyze the squared mean prediction error conditional on the design on the event $\mathcal{A}_d\cap \mathcal{E}_d$. First, by the projection theorem, we have
\begin{align*}
&\|\hat f_d-f\|_{L^2(\mathbb{P}^X)}^2=\|\hat f_d-P_{\le d}f\|_{L^2(\mathbb{P}^X)}^2+\|P_{> d}f\|_{L^2(\mathbb{P}^X)}^2.
\end{align*}
The last term is exactly the oracle bias term. Moreover, by Lemma \ref{LemChangeNorm1.main} (applied to the observations $X_i'$ taking values in the $d$-dimensional Hilbert space $U_d$ and $\mu=0$), we have 
\begin{align*}
\|\hat f_d-P_{\le d}f\|_{n}^2\le  \frac{3}{2}\|\hat f_d-P_{\le d}f\|_{L^2(\mathbb{P}^X)}^2\le 3\|\hat f_d-P_{\le d}f\|_{n}^2.
\end{align*}
Letting $\hat\Pi_n:\mathbb{R}^n\rightarrow S_nU_d,\mathbf{y}\mapsto \argmin_{g\in  U_d}\|\mathbf{y}-S_ng\|_n^2$ be the orthogonal projection onto $S_nU_d$, we have, on the event $\mathcal{A}_d\cap \mathcal{E}_d$,
\begin{align*}
&\mathbb{E}(\|\hat f_d-P_{\le d}f\|_n^2\,|\,X_1,\dots,X_n) \\
&= \|\hat\Pi_n f-P_{\le d}f\|_n^2+\mathbb{E}(\|\hat\Pi_n\boldsymbol{\epsilon}\|_n^2\,|\,X_1,\dots,X_n)\\
&=\|\hat\Pi_n f-P_{\le d}f\|_n^2+\sigma^2dn^{-1}\le  \| P_{> d}f\|_{n}^2+\sigma^2dn^{-1}.
\end{align*}
By Assumption \ref{SubGauss}, $\langle X,P_{>d}f\rangle$ is sub-Gaussian with factor $\SGN \|P_{>d}f\|_{L^2(\mathbb{P}^X)}$. By Bernstein's inequality (see, e.g., Theorem 2.8.1 and Lemma 2.7.6 in \cite{MR3837109}), we get that 
\begin{align*}
\| P_{> d}f\|_{n}^2-\| P_{> d}f\|_{L^2(\mathbb{P}^X)}^2\le CL^2\| P_{> d}f\|_{L^2(\mathbb{P}^X)}^2
\end{align*}
with probability at least $1-e^{-n}$. The claim follows from collecting all these inequalities.\qed

\subsection{Proof of Theorem \ref{MainResultClassical}}\label{MainResultClassicalProof}

The proof is based on combining Lemma \ref{LemBVDecCond} with the perturbation bounds from Section \ref{SecTools}. Let $c_1$ and $c_2$ be the minima of the respective constants in Lemmas \ref{ThmEvDev} and \ref{LemConcIneq}.
First, consider the variance term $V_d$. By \eqref{MainResultClassicalCond} and Corollary~\ref{CorChangeNormEV}, we have 
\begin{align*}
\forall j\le d,\quad \|\hat u_j\|_{L^2(\mathbb{P}^X)}^2& \le 4\lambda_j
\end{align*}
with probability at least $1-e^{-c_2n}$, provided that $c\le c_1/4$. Moreover, by \eqref{MainResultClassicalCond}, Corollary~\ref{CorEVConcLB} and the union bound, we have
\begin{equation*}
\forall j\le d,\quad  \hat\lambda_j\ge   \lambda_j/2
\end{equation*}
with probability at least $1- cne^{1-c_2n/4}$. Combining both inequalities, we get $V_d\le 8\sigma^2 d/n$ with probability at least $1- 2cne^{1-c_2n/4}$. Next, consider the bias term $B_d^2(f)$. Corollary~\ref{CorChangeNormFunction} implies
\begin{align*}
    B_d^2(f) = \| \hat P_{>d}f \|^2_{L^2(\mathbb{P}^X)} \le 2 \| \hat P_{>d}f \|^2_n + 2 \frac{\tr_{>d}(\Sigma)}{c_1n} \| \hat P_{>d}f \|^2
\end{align*}
with probability at least $1-e^{-c_2n}$. Observe that $\| \hat P_{>d}f \|^2_n \le \hat{\lambda}_{d+1} \| f \|^2$ and $\| \hat P_{>d}f \|^2 \le \| f \|^2$. Using also Corollary \ref{CorChangeNormEV} and \eqref{MainResultClassicalCond}, we get $B_d^2(f)\le 4\lambda_{d+1}\| f \|^2$ with probability at least $1-e^{-c_2n}$, provided again that $c\le c_1/4$. This completes the proof. \qed

\subsection{Proof of Theorem \ref{BenignOverfittingResult}} \label{ProofBenignOverfittingResult}

Similarly as in the proof of Theorem \ref{MainResultClassical}, consider first the variance term $V_d$. Observe that
\begin{align}
    V_d
    &= \frac{\sigma^2}{n} \sum_{j = 1}^{j^*} \frac{\|\hat{u}_j\|^2_{L^2(\mathbb{P}^X)}}{\hat{\lambda}_j} + \frac{\sigma^2}{n} \sum_{j = j^*+1}^d \frac{\tr(\hat{P}_j \Sigma)}{\hat{\lambda}_j} \nonumber\\
    &\le \frac{\sigma^2}{n} \sum_{j = 1}^{j^*} \frac{\|\hat{u}_j\|^2_{L^2(\mathbb{P}^X)}}{\hat{\lambda}_j} + \frac{\sigma^2}{n} \frac{1}{\hat\lambda_n}\sum_{j = j^*+1}^d \tr(\hat{P}_j \Sigma)\nonumber\\
    &\le \frac{\sigma^2}{n} \sum_{j = 1}^{j^*} \frac{\|\hat{u}_j\|^2_{L^2(\mathbb{P}^X)}}{\hat{\lambda}_j} + \frac{\sigma^2}{n} \frac{1}{\hat{\lambda}_n} \bigg(\mathcal{E}_{\le j^*}^{PCA}(0) + \sum_{j=j^*+1}^d \tr(\hat{P}_j \Sigma_{>j^*})\bigg),\label{eq:variance:benign:decomp}
\end{align}
where we used 
\begin{align*}
    \|\hat{u}_j\|^2_{L^2(\mathbb{P}^X)}=\langle \hat u_j,\Sigma\hat u_j\rangle=\tr(\hat{P}_j \Sigma)
\end{align*}
in the first equality and
\begin{align*}
    \lambda_k\sum_{j=j^*+1}^d\tr(P_k\hat{P}_j )\le \lambda_k\tr(P_k\hat{P}_{>j^*} )=\lambda_k\|P_k\hat{P}_{>j^*}\|_2^2,\qquad k\le j^*,  
\end{align*}
and the definition of $\mathcal{E}_{\le j^*}^{PCA}(0)$ in the last inequality.
Let $c_2$ be the minimum of the constants in Lemmas \ref{ThmEvDev}, \ref{LemBartlett}, \ref{LemConcIneq} and \ref{LemExcessRiskHighProbability}. Lemma \ref{LemBartlett} and \eqref{eq:change:remainder:trace} yield
\begin{align}
    \hat{\lambda}_n \ge \frac{1}{2} \frac{\tr_{>n}(\Sigma)}{n} \ge \frac{1}{4} \frac{\tr_{>j^*}(\Sigma)}{n} \label{ineq:lower:bound:lambda_n:hat}
\end{align}
with probability at least $1-e^{1-c_2n}$, while Corollary \ref{CorExcessRisk0} implies, for all $1\le t \le c_2n/4$,
\begin{align}
    {\mathcal{E}}_{\le j^*}^{PCA}(0) \le C t^2 \frac{j^* \tr_{>j^*}(\Sigma)}{n} \label{ineq:bound:excess:risk}
\end{align}
with probability at least $1-3e^{1-t}$. Moreover, Corollaries \ref{CorEVConcLB} and \ref{CorChangeNormFunction} and the union bound yield, for all $j\le j^*$,
\begin{align}
    \hat{\lambda}_j^{-1}\|\hat{u}_j\|^2_{L^2(\mathbb{P}^X)} \le \hat{\lambda}_j^{-1}\Big(2\hat{\lambda}_{j}+2\frac{ \tr_{>j^*}(\Sigma)}{c_1n}\Big) \le 2+4 \frac{\tr_{>j^*}(\Sigma)}{c_1n \lambda_{j^*}} \le
    2+\frac{4B}{c_1} \label{ineq:bound:first:sum}
\end{align}
with probability at least $1-cne^{1-c_2n/4}-e^{-c_2n}$, where $c_1$ is the constant in Lemma \ref{LemConcIneq}. Inserting \eqref{ineq:lower:bound:lambda_n:hat}--\eqref{ineq:bound:first:sum} into \eqref{eq:variance:benign:decomp}, we arrive at
\begin{align}
    V_d\le C\Big(t^2\frac{\sigma^2j^*}{n}+\frac{\sigma^2}{\tr_{>j^*}(\Sigma)}\sum_{j=j^*+1}^d \tr(\hat{P}_j \Sigma_{>j^*})\Big) \label{ineq:upper:bound:variance}
\end{align}
with probability at least $1-Ce^{-ct}$, where $c,C>0$ are constants depending only on $c_1,c_2$ and $B$. Next, consider the bias term $B_d^2(f)$. Arguing as in the proof of Theorem~\ref{MainResultClassical}, Corollaries \ref{CorChangeNormFunction} and \ref{CorChangeNormEV} and the fact that $j^*\le cn$ imply that
\begin{align*}
    B_d^2(f)\le B_{j^*}^2(f)  \le 3 \lambda_{j^*+1}\| f \|^2 + 4 \frac{\tr_{>j^*}(\Sigma)}{c_1n} \| f \|^2\le \Big(\frac{3}{B}+\frac{4}{c_1}\Big)\frac{\tr_{>j^*}(\Sigma)}{n}\| f \|^2
\end{align*}
with probability at least $1-e^{-c_2n}$. This completes the proof of the first claim.

It remains to prove the second claim. Observing that $\tr(\hat{P}_j \Sigma_{>j^*}) \ge 0$ for all $j$, we have
\begin{align}
    \sum_{j=j^*+1}^n \tr(\hat{P}_j \Sigma_{>j^*}) &\le \frac{1}{\hat{\lambda}_n} \sum_{j=1}^n \tr(\hat{\lambda}_j \hat{P}_j \Sigma_{>j^*}) = \frac{1}{\hat{\lambda}_n} \tr(\hat{\Sigma} \Sigma_{>j^*})\label{ineq:bound:second:part:second:sum}
\end{align}
and
\begin{align*}
    \tr(\hat{\Sigma} \Sigma_{>j^*}) = \sum_{j>j^*} \lambda_j^2 \Big( \frac{1}{n} \sum_{i=1}^n \eta_{ji}^2\Big)
\end{align*}
with Karhunen-Loève coefficients $$\eta_{ji}=\lambda_j^{-1/2} \langle X_i,u_j \rangle,\quad j \ge 1,1\le i\le n.$$ 
By Assumption \ref{SubGauss}, we have $\sup_{j\ge 1} \| \eta_{ji} \|_{\psi_2} \le L$ for all $i$. This implies by centering (see, e.g., Remark 5.18 in \cite{V12})
\begin{align*}
    \| \eta_{ji}^2-1 \|_{\psi_1} \le 2 \| \eta_{ji}^2 \|_{\psi_1} \le 4 \| \eta_{ji} \|_{\psi_2}^2 \le 4L^2.
\end{align*}
Thus, $(\eta_{ji}^2-1)$ is a sequence of real-valued centered random variables such that $(\eta_{ji}^2-1)_{i=1}^n$ are i.i.d.~for each $j > j^*$ and such that $\sup_{j>j^*} \|\eta_{ji}^2-1\|_{\psi_1} \le 4L^2$. By the second bound in Lemma \ref{LemConcWeightedNorm} below applied with $t=n$, we have
\begin{align*}
    \sum_{j > j^*} \lambda_j^2 \Big( \frac{1}{n} \sum_{i=1}^n  \eta_{ji}^2 - 1  \Big) \le 2C  \tr_{>j^*}(\Sigma^2)
\end{align*}
with probability at least $1-e^{-n}$.
Hence, we get
\begin{align}
    \tr (\hat{\Sigma} \Sigma_{>j^*}) \le (1+2C)\tr_{>j^*}(\Sigma^2) \label{ineq:bound:tr:Sigma_hat*Sigma>j*}
\end{align}
with probability at least $1-e^{-n}$. Combining \eqref{ineq:lower:bound:lambda_n:hat}, \eqref{ineq:bound:second:part:second:sum} and \eqref{ineq:bound:tr:Sigma_hat*Sigma>j*}, the second claim follows. \qed

\subsection{Proof of Lemma \ref{LemExcessRiskHighProbability}}
\label{ProofLemExcessRiskHighProbability}

\begin{lemma} \label{lemma_excess_risk_bounds_det} Let $d\ge  1$ and $\mu\ge \lambda_{d+1}$. Let $r\le d$ be the largest integer such that $\lambda_r>\mu$ (set $r=0$ if $\mu\ge \lambda_1$) and let $S_{\le r}=\sum_{j\le r} (\lambda_j - \lambda_{d+1})^{-1/2}P_j$. Then, on the joint event
 \begin{align}\label{EventRW}
 \{\|S_{\le r}(\hat \Sigma - \Sigma)S_{\le r}\|_{\infty} \le 1/4\} \cap \{\hat{\lambda}_{d+1} - \lambda_{d+1} \le (\lambda_r- \lambda_{d+1})/2\},
\end{align}
we have
\begin{align*}
\mathcal E^{PCA}_{\le d}(\mu)\le 16\sum_{j\le r}\frac{\|P_j (\hat \Sigma - \Sigma) P_{>r}\|_2^2}{\lambda_j-\lambda_{d+1}}.
\end{align*}
\end{lemma}

\begin{proof}
By construction, we have
\begin{align*}
\mathcal E^{PCA}_{\le d}(\mu)\le \sum_{j\le r}(\lambda_j-\mu)\|P_j\hat P_{>d}\|_2^2\le \sum_{j\le r}(\lambda_j-\lambda_{d+1})\|P_j\hat P_{>d}\|_2^2.
\end{align*}
The last term is analyzed in the proof of Proposition 3.5 in \cite{RW17} (set $s=r$ and $\mu=\lambda_{d+1}$), where it is shown that 
\begin{align*}
\sum_{j\le r}(\lambda_j-\lambda_{d+1})\|P_j\hat P_{>d}\|_2^2\le 16\sum_{j\le r}\frac{\|P_j(\hat \Sigma - \Sigma)P_{>r}\|_2^2}{\lambda_j-\lambda_{d+1}},
\end{align*}
provided that \eqref{EventRW} holds.
\end{proof}

\begin{lemma}\label{LemConcHSNorm}
Grant Assumption \ref{SubGauss}. Then, for any $t\ge  1$, we have
\begin{align*}
\sum_{j\le r}\frac{\|P_j(\hat \Sigma - \Sigma)P_{>r}\|_2^2}{\lambda_j-\lambda_{d+1}}\le C\frac{t^2}{n}\sum_{j\le r}\frac{\lambda_j\operatorname{tr}_{>r}(\Sigma)}{\lambda_j-\lambda_{d+1}}
\end{align*}
with probability at least $1-e^{-t}$.
\end{lemma}

\begin{proof}
We have
\begin{align*}
\sum_{j\le r}\frac{\|P_j(\hat \Sigma - \Sigma)P_{>r}\|_2^2}{\lambda_j-\lambda_{d+1}}=\sum_{j\le r}\sum_{k>r}\frac{\lambda_j\lambda_k}{\lambda_j-\lambda_{d+1}}\Big(\frac{1}{n}\sum_{i=1}^n\eta_{ji}\eta_{ki}\Big)^2,
\end{align*}
where $\eta_{ji}=\lambda_j^{-1/2}\langle X_i,u_j\rangle$. The claim now follows from the first bound in Lemma~\ref{LemConcWeightedNorm} below by proceeding similarly as in the proof of Theorem \ref{BenignOverfittingResult}.
\end{proof}

We now finish the proof of Lemma \ref{LemExcessRiskHighProbability}. Applying Lemma \ref{ThmEvDev} with $j=d+1$ and $y=(\lambda_{r}-\lambda_{d+1})/(2\lambda_{d+1})$, we get 
\begin{align*}
\mathbb{P}(\hat{\lambda}_{d+1} - \lambda_{d+1} > (\lambda_{r}- \lambda_{d+1})/2)&\le \exp\big(1-c_2n(y \wedge y^2)\big) \\
&\le \exp\Big(1-\frac{c_2}{4} \frac{n(\lambda_r-\lambda_{d+1})^2}{\lambda_r^2}\Big),
\end{align*}
provided that 
\begin{align*}
\max\Big(\frac{2\lambda_{d+1}}{\lambda_{r}-\lambda_{d+1}},1\Big)\sum_{k>d+1}\frac{\lambda_k}{\lambda_{d+1}-\lambda_{k}+y\lambda_{d+1}}\le c_1n.
\end{align*}
This is the case if
\begin{align*}
    \frac{\lambda_r}{\lambda_r-\lambda_{d+1}} \sum_{k>d+1} \frac{\lambda_k}{\lambda_r-\lambda_k} \le c_1n/4.
\end{align*} 
Moreover, proceeding as in the proof of Lemma \ref{LemConcIneq}, there are constants $c_1, c_2 \in (0,1)$ such that
\begin{align*}
\mathbb{P}(\|S_{\le r}(\hat \Sigma - \Sigma)S_{\le r}\|_{\infty} > 1/4)\le \exp\Big(1-c_2\frac{n(\lambda_r-\lambda_{d+1})^2}{\lambda_r^2}\Big),
\end{align*}
provided that 
\begin{align*}
\frac{\lambda_r}{\lambda_r-\lambda_{d+1}}\sum_{j\le r}\frac{\lambda_j}{\lambda_j-\lambda_{d+1}}\le c_1n.
\end{align*}
Combining these concentration inequalities with Lemmas \ref{lemma_excess_risk_bounds_det} and \ref{LemConcHSNorm}, the claim follows. \qed


\section*{Acknowledgements}
We are grateful for the helpful comments by the referee. The research of MW has been partially funded by Deutsche Forschungsgemeinschaft (DFG)~- SFB1294 - 318763901. LH was supported by Deutsche Forschungsgemeinschaft (DFG) - FOR5381 - 460867398.


\bibliographystyle{plain}
\bibliography{references.bib}


\appendix


\section{Concentration inequalities}

The following lemma provides two  concentration inequalities for (squared) sums of sub-exponential random variables, adapted for our purposes.

\begin{lemma}\label{LemConcWeightedNorm}
Let $(a_j)$ be a sequence of nonnegative real numbers such that $\|a\|_1=\sum_{j=1}^\infty a_j<\infty$. Let $(X_{ij})$ be a sequence of real-valued centered random variables with $j\ge 1$ and $1\le i\le n$ such that $(X_{ij})_{i=1}^n$ are i.i.d.~for each $j\ge 1$ and such that $\sup_{j\ge 1}\|X_{1j}\|_{\psi_1}=L<\infty$. Then there exists a constant $C$ depending only on $L$ such that, for every $t\ge  1$, we have
\begin{align*}
\sum_{j\ge 1}a_j\Big(\frac{1}{n}\sum_{i=1}^nX_{ij}\Big)^2\le C\|a\|_1\frac{t^2}{n}
\end{align*}
and
\begin{align*}
\sum_{j\ge 1}a_j\Big(\frac{1}{n}\sum_{i=1}^nX_{ij}\Big)\le C\|a\|_1\Big(\frac{\sqrt{t}}{\sqrt{n}}+\frac{t}{n}\Big)
\end{align*}
each with probability at least $1-e^{-t}$.
\end{lemma}

\begin{proof}
For $p\ge  1$, Minkowski's inequality yields
\begin{align}
&\mathbb{E}^{1/p}\bigg(\sum_{j\ge 1}a_j\Big(\frac{1}{n}\sum_{i=1}^nX_{ij}\Big)^2\bigg)^p\le \sum_{j\ge 1}a_j\mathbb{E}^{1/p}\Big(\frac{1}{n}\sum_{i=1}^nX_{ij}\Big)^{2p}\label{EqLemConcHSNorm}.
\end{align}
Bernstein's inequality (see, e.g., Theorem 2.8.1 in \cite{MR3837109}) combined with Theorem 2.3 in \cite{MR3185193} yields
\begin{align} \label{EqLemConcLpNorm}
\forall p\ge  1,\qquad\mathbb{E}^{1/p}\Big(\frac{1}{n} \sum_{i=1}^nX_{ij}\Big)^{2p} &\le C(p!\, n^{-p} + (2p)!\, n^{-2p})^{1/p} \\
&\le C\Big(\sqrt{\frac{p}{n}}\vee\frac{p}{n}\Big)^2\le 
C\frac{p^2}{n} \notag
\end{align} 
with a constant $C$ depending only on $L$.
Inserting this into \eqref{EqLemConcHSNorm}, we get
\begin{align*}
\forall p\ge  1,\qquad\mathbb{E}^{1/p}\bigg(\sum_{j\ge 1}a_j\Big(\frac{1}{n}\sum_{i=1}^nX_{ij}\Big)^2\bigg)^p\le Cp^2\frac{\|a\|_1}{n},
\end{align*}
meaning that 
\begin{align}
\bigg\|\sum_{j\ge 1}a_j\Big(\frac{1}{n}\sum_{i=1}^nX_{ij}\Big)^2\bigg\|_{\psi_{1/2}}\le C\frac{\|a\|_1}{n}. \label{EqLemConcHSNorm2}
\end{align}
Now, for a nonnegative random variable $Y$ with $\|Y\|_{\psi_{1/2}}<\infty$, we have $\|Y^{1/2}\|_{\psi_1}^2\le \|Y\|_{\psi_{1/2}}$, as can be seen from the definition of the $\psi_\alpha$ norms in Section \ref{BasicNotation} and the Cauchy-Schwarz inequality. In particular, we have $\mathbb{P}(Y\ge C \|Y\|_{\psi_{1/2}}t^2)\le e^{-t}$ for every $t\ge 1$ and an absolute constant $C>0$. Combining this with \eqref{EqLemConcHSNorm2}, the first claim follows.

For the second claim, Minkowski's inequality and \eqref{EqLemConcLpNorm} yield
\begin{align*}
    \forall p \ge 1, \qquad \mathbb{E}\bigg(\sum_{j\ge 1}a_j\Big(\frac{1}{n}\sum_{i=1}^nX_{ij}\Big)\bigg)^{2p}\le C^{p} \| a \|_1^{2p} (p! \, n^{-p} + (2p)! \, n^{-2p}). 
\end{align*}
We obtain by the second part of Theorem 2.3 in \cite{MR3185193} that
\begin{align*}
    \sum_{j\ge 1}a_j\Big(\frac{1}{n}\sum_{i=1}^nX_{ij}\Big) &\le 3\sqrt{C} \| a \|_1 \Big(\sqrt{\Big( \frac{1}{n}+\frac{1}{n^2} \Big)t } + \frac{t}{n} \Big) \\
    &\le 3 \sqrt{2C}\|a\|_1\Big(\frac{\sqrt{t}}{\sqrt{n}}+\frac{t}{n}\Big)
\end{align*}
with probability at least $1-e^{-t}$.
\end{proof}

\end{document}